\documentclass{amsart}
\usepackage{enumerate}
\usepackage{tikz}
\usetikzlibrary{automata,positioning}
\usepackage{hyperref}
\usepackage{mathrsfs}
\usepackage{amsfonts}
\usepackage{latexsym}
\usepackage{epsfig}
\usepackage{indentfirst}
\graphicspath{{figures/}}
\usepackage{amsmath}
\usepackage{amsthm}
\usepackage{graphicx}
\usepackage{amssymb}
\usepackage{amscd}
\usepackage{latexsym}
\usepackage{color}
\DeclareGraphicsExtensions{.eps}
\input xy
\xyoption{all}
\usepackage{fancyhdr}       

\newtheorem{prop}{Proposition}
\newtheorem{cor}{Corollary}
\newtheorem{defni}{Definition}
\newtheorem{theor}{Theorem}
\newtheorem{lem}{Lemma}

\begin{document}
\title{d invariants obstruction to sliceness of  a class of algebraically slice knots
}
\author{Chen Zhang}
\address{Michigan State University}
\email{zhan1339@msu.edu}
\thanks{}

\begin{abstract}
We use d invariants to show nonsliceness of a set of algebraically slice knots.
\end{abstract}
\maketitle

\section{Introduction}
In \cite{rudolph1976independent}, Rudolph asks whether the set of algebraic knots are linearly independent in the knot concordance group $\mathcal{C}$. An algebraic knot is, by definition, the connected link of an isolated singularity of a polynomial map $f:{\mathbb{C}^{2}\rightarrow\mathbb{C}}$. It can also be defined as an iterated torus knot $T_{p_{1},q_{1};\cdots;p_{n},q_{n}}$ with indices satisfying $p_{i},q_{i}>0$ and $q_{i+1}>p_{i}q_{i}p_{i+1}$.

A knot is algebraically slice if it is in the kernel of Levine's classifying homomorphism \cite{levine1969invariants}. Livingston and Melvin \cite{livingston1983algebraic} observed that, for any knot $K$,
\begin{center}
    $K_{p,q_{1}}\#T_{p,q_{2}}\#-K_{p,q_{2}}\#-T_{p,q_{1}}$
\end{center}
is an algebraically slice knot. Here, $K_{p,q}$ is the $(p,q)$-cable of $K$.

 When $K$ is an algebraic knot, all the components in the above connected sum, up to mirror images, are algebraic knots provided $q_{i}$'s are large enough. Since sliceness of a knot implies it is algebraically slice, it is interesting to ask when the knot above is slice. In \cite{kirk2012non}, Hedden, Kirk and Livingston used Casson-Gordon invariants \cite{casson1986cobordism} to show that
\begin{theor}\cite{kirk2012non}
    For appropriately chosen integers $q_{i}$,
    \begin{center}
        $T_{2,3;2,q_{n}}\#T_{2,q_{1}}\#-T_{2,3;2,q_{1}}\#-T_{2,q_{n}}$
    \end{center}
    are not slice.
\end{theor}

In this paper, we use the d invariants obstruction from \cite{hedden2012topologically} to show that for any L space knot $K$,
\begin{theor}
    $K_{2,k_{1}}\#-T_{2,k_{1}}\#-{K}_{2,k_{2}}\#T_{2,k_{2}}$ has infinite order in the knot concordance group $\mathcal{C}$ when $k_{1}$ and $k_{2}$ are pair of distinct prime numbers greater than 3.
\end{theor}

More specifically, we will pass to the 2-fold branched cover: $M=\Sigma_{2}(K_{2,k_{1}}\#-T_{2,k_{1}})$. Using a knot surgery description of $M$ and Ni-Wu's formula\cite{ni2015cosmetic} of d invariants, we can show that there doesn't exist a metabolizer of $H_{1}(M)$ such that the relative d invariants vanish on it. Then the result follows from the obstruction in the next section.

As a corollary, we have
\begin{cor}
    $K_{2,k_{1}}\#-T_{2,k_{1}}\#-{K}_{2,k_{2}}\#T_{2,k_{2}}$ has infinite order in the knot concordance group $\mathcal{C}$ when $K_{2,k_{1}}$ and $K_{2,k_{2}}$ are algebraic knots.
\end{cor}

Note that all results about the nonsliceness of $K_{2,k_{1}}\#-T_{2,k_{1}}\#-{K}_{2,k_{2}}\#T_{2,k_{2}}$ have used Casson-Gordon invariants and our result is the first one which uses Heegaard Floer homology.

\paragraph*{\textbf{Acknowledgment}} The author wishes to thank his advisor Matt Hedden for his guidance and help.

\section{d invariants obstruction}
In this section, we review an obstruction for a knot to be slice from \cite{hedden2012topologically}. Let $K$ be a knot in $S^{3}$ and let $\Sigma(K)$ be the 2-fold branched cover of $K$. Suppose $K$ is slice, which means $K$ bounds a smooth disk in $D^{4}$. Then $\Sigma(K)$ bounds a $\mathbb{Z}/2\mathbb{Z}$-homology 4-ball $W$. Hence, to show the nonsliceness of $K$, it is enough to show the nonexistence of $W$.

The linking form on $\Sigma(K)$ provides an initial constraint on the pair $(W,\Sigma(K))$.  To state it, recall that a subgroup $M\subset H_{1}(\Sigma(K))$ is called a $metabolizer$ if\\
$\bullet$ ${|M|}^{2}=|T_{1}(Y)|$, where $T_{1}$ denotes the torsion subgroup of $H_{1}(Y)$, and\\
$\bullet$ The $\mathbb{Q}/\mathbb{Z}$-valued linking form on $H_{1}(Y)$ is identically zero on $M$.

If $K$ is slice, then the $\mathbb{Z}/2\mathbb{Z}$-homology ball bounded by $\Sigma(K)$ gives rise, via the kernel of the inclusion induced map, $i_{*}:H_{1}(\Sigma(K))\rightarrow H_{1}(W)$, to a metabolizer of $H_{1}(\Sigma(K))$\cite{casson1986cobordism}. This linking form can often be used to obstruct sliceness. Note, however, that when $K$ is algebrically slice, this obstruction vanishes. 

With Heegaard Floer homology, we get additional structures on the $metabolizer$. Recall that, for a rational homology 3-sphere $Y$, the Heegaard-Floer homology of $Y$ splits with respect to $\rm{Spin}^{c}$ structures over $Y$,
\begin{center}
    $HF^{+}(Y)=\bigoplus\limits_{t\in\rm{Spin}^{c}(Y)}{HF}^{+}(Y,\mathfrak{t})$.
\end{center}

For each $\mathfrak{t}\in\rm{Spin}^{c}(Y)$, ${HF}^{+}(Y,\mathfrak{t})$ is an $\mathbb{F}[U]$ module and the d-invariants are defined as follows,

\begin{defni}
    $d(Y,\mathfrak{s})={\rm{min}}_{\alpha\neq0\in{HF}^{+}(Y,\mathfrak{s})}\{gr(\alpha)|\alpha\in\rm{Im}U^{k},$ $\mathrm{for}$ $\mathrm{all}$ $k\geq0\}$
\end{defni}

The d invariants satisfy the following two properties,\\
1.(Additivity) $d(Y\#Y^{'},\mathfrak{s}\#{\mathfrak{s}}^{'})=d(Y,\mathfrak{s})+d(Y^{'},\mathfrak{s}^{'})$: that is, $d$ is additive under connected sums.\\
2\label{d2}.(Vanishing) Suppose $(Y,\mathfrak{s})=\partial(W,\mathfrak{t})$, where $W$ is a $\mathbb{Q}$-homology ball and $\mathfrak{t}$ is a $\rm{Spin}^{c}$ structure on $W$ that restricts to $\mathfrak{s}$ on $Y$. Then $d(Y,\mathfrak{s})=0$.

The obstruction is defined as a difference of correction terms.

\begin{defni}
    For $Y$ a $\mathbb{Z}/2\mathbb{Z}$-homology sphere, define the relative d invariants as $\Bar{d}(Y.\mathfrak{s})=d(Y,\mathfrak{s})-d(Y,\mathfrak{s}_{0})$, where $\mathfrak{s}_{0}$ is the unique spin structure on $Y$.
\end{defni}

When $H_{1}(Y;\mathbb{Z}/2\mathbb{Z})=0$, Poincare duality and the Chern class provide a bijection ${\mathrm{Spin}}^{c}\longleftrightarrow H_{1}(Y)$. Combining with the d invariants property \ref{d2} above, we see that the d invariants vanish on a metabolizer. One can package this using the following\cite{hedden2012topologically},

\begin{theor}
    Let $P$ be a finite set of (distinct) odd primes. Suppose that $W$ is a $\mathbb{Z}/2\mathbb{Z}$-homology 4-ball and $\partial W=\#_{p\in P}Y_{p}\#Y_{1}$, where\\
    $\bullet$ $p^{k}H_{1}(Y_{p})=0$ for each $p\in P$ and some $k\geq0$.\\
    $\bullet$ $Y_{1}$ is a $\mathbb{Z}$-homology 3-shpere.\\
    Then for each $p\in P$, there is a $metabolizer$ $M_{p}\subset H_{1}(Y_{p})$ for which $\Bar{d}(Y_{p},\mathfrak{s}_{m_{p}})=0$ for all $m_{p}\in M_{p}$.
\end{theor}

By using branched covers, the theorem yields the desired concordance obstruction.

\begin{cor}\label{Core cor}
    Let $K=\#_{p\in P}K_{p}\#K_{1}$ be a connected sum of knots satisfying\\
    $\bullet$ $p^{k}H_{1}(\Sigma(K_{p}))=0$ for each $p$ in a set of primes, $P$, and some $k$,\\
    $\bullet$ $H_{1}(\Sigma(K_{1}))=0$.\\
    Suppose $K$ is slice.Then for each $p\in P$, there is a $metabolizer$ $M_{p}\subset H_{1}(\Sigma(K_{p}))$ for which $\Bar{d}(\Sigma(K_{p}),\mathfrak{s}_{m_{p}})=0$ for all $m_{p}\in M_{p}$.
\end{cor}

Note that Corollary \ref{Core cor} shows more; normaly the linear combination $\sum K_{p}$ isn't concordant to any knot with $\mathrm{det}(K)=1$.

\section{Topology of the 2-fold branched cover}\label{top}
In this section, we use the algorithm from \cite{dai2022rank} to give a knot surgery description of $\Sigma_{2}(K_{2,p})$.

We first review the notion of $rational$ $unknotting$ $number$ $one$ $patterns$. For the definition of rational tangle and the bijection between the rational tangles in a fixed 3-ball $B^{3}$ and $\mathbb{Q}\cup\{\infty\}$, one can refer to section 2.1 in \cite{dai2022rank}.

\begin{defni}
Let $P\subseteq S^{1}\times D^{2}$ be a pattern. We say that P has a rational unknotting number one if there exists a rational tangle T in P such that replacing T with another rational tangle $T^{'}$ gives a knot which is unknotted in the solid torus. We say that P has proper rational unknotting number one if $T^{'}$ can be taken to be a proper tangle replacement: that is, connecting the same two pairs of marked points as T.
\end{defni}

For a rational unknotting number one pattern P, we have
\begin{center}
    $\Sigma_{2}(P(U))\cong {S}^{3}_{p/q}(J)$
\end{center}
for some strongly invertible knot $J$ and surgery coefficient $p/q$. The claim is immediate from the Montesinos trick: since $P^{'}$ is an unknot, the branched double cover over $P^{'}$ is $S^{3}$. The 3-ball $B^{3}$ containing $T{'}$ lifts to a solid torus in $S^{3}$, and replacing $T{'}$ with $T$ corresponds to doing surgery on the core of this solid torus. Moreover, we can explicitly produce J and the surgery coefficient $p/q$. Here, we will use $T_{2,k}$ to illustrate the procedure, which is given in Figure \ref{figl}. For the general case, one can refer to \cite{dai2022rank}.

Replacing $T$ with $T^{'}$ gives an unknot in the solid torus. Let $\gamma$ be a reference arc in $B^{3}$ which has one endpoint on each component of $T{'}$, displayed in panel (2). Taking 2-fold branched cover of $B^{3}$ along $T{'}$ gives a solid torus and the lift of $\gamma$ in the solid torus is the core of this solid torus, i.e. $J$. Let $F_{t}$ be an isotopy of the solid torus moving $P^{'}$ into a local unknot in $S^{1}\times D^{2}$. We then cut along the disk bounded by the unknot and glue two copies of the disk complement to get the 2-fold branched cover. We also keep track of $\gamma$ along $F_{t}$ and lift it to the 2-fold branched cover. This gives the desired strongly invertible knot $J$, which in this case is just an unknot, displayed in panel (4).

\begin{figure}[ht]
 
\centering
\includegraphics[scale=0.5]{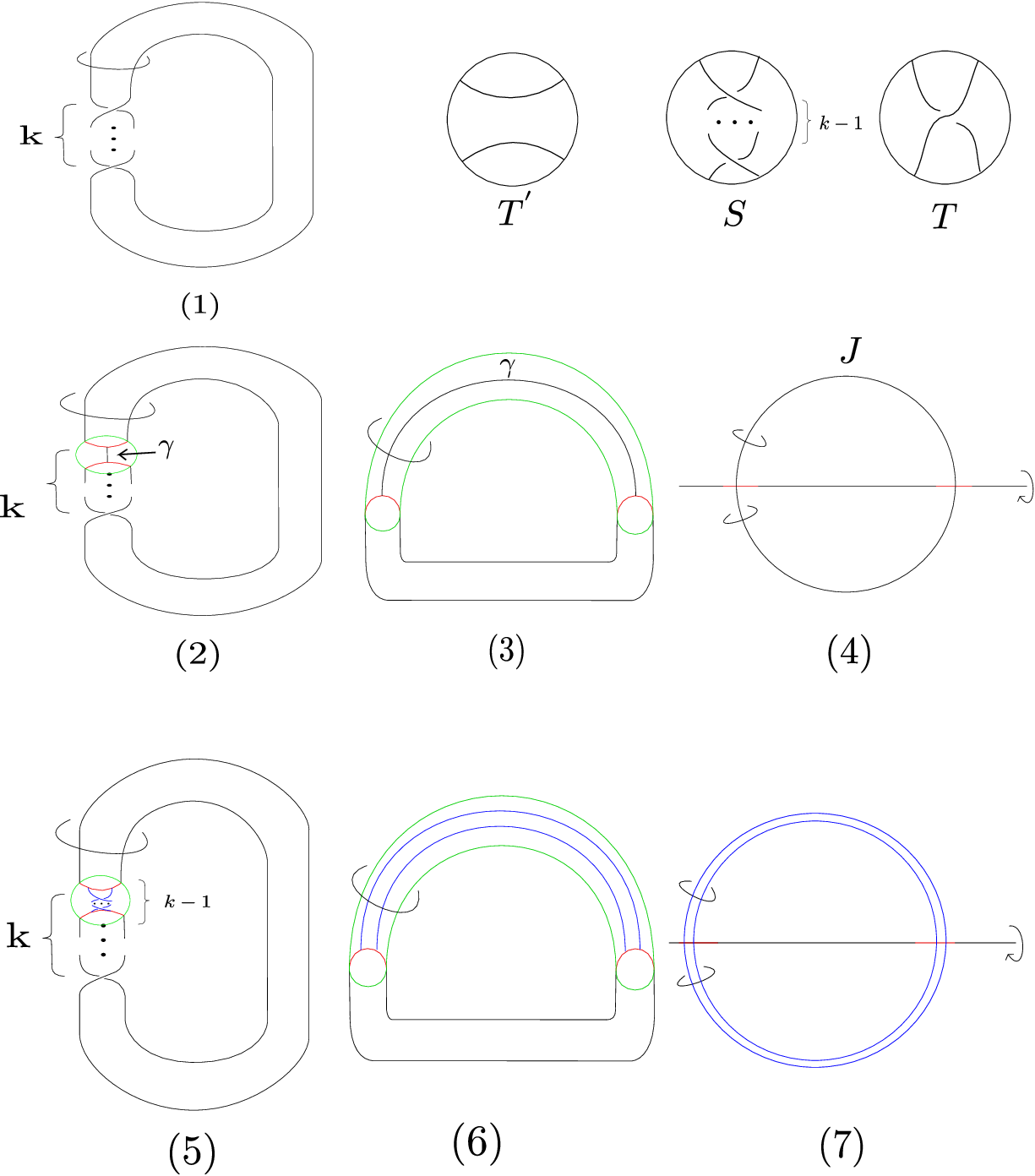}
\caption{$J$ and the surgery coefficient}
\label{figl}
 
\end{figure}

To compute the surgery coefficient $p/q$, we must find the unique rational tangle $S$ in $B^{3}$ which lifts to a pair of $\tau$-equivariant Seifert framings of $J$, which can be done by running $F_{t}$ backwards. We first find a $\tau$-invariant Seifert framing of $J$ in the 2-fold branched cover, then quotient it by $\tau$ and reverse the isotopy $F_{t}$ to draw $S$ in the original 3-ball $B^{3}$. Since $J$ is a $\tau$-equivariant unknot, we can pick a parallel copy of $J$ to be Seifert framing. Quotienting this pair by $\tau$ gives us a pair of arcs. Keeping track of $F_{t}$ backwards, it adds $k-1$ negative crossings to the pair of arcs. Hence, the rational tangle with $k-1$ negative crossings is the desired rational tangle in $S$ in $B^{3}$. By the Montesinos trick, the surgery coefficient $p/q$ is then precisely the rational number identified with the original tangle $T$ relative to the choice of reference tangles $T_{\infty}=T^{'}$ and $T_{0}=S$. In our example, the surgery coefficient is $k$. From the discussion above, we have
\begin{center}
    $\Sigma_{2}(T_{2,k})\cong S^{3}_{k}(U)$.
\end{center}

 Let $K$ be an oriented knot in $S^{3}$. We can now extend the discussion above to the branched cover of a cable knot $K_{2,k}$. Recall that $K_{2,k}$ can be constructed by taking the image of $T_{2,k}$ inside the gluing
 \begin{center}
     $S^{3}\cong (S^{3}-N(\mu)){\cup}_{\partial N(\mu)}(S^{3}-N(K))$
 \end{center}
 formed by a boundary identification which maps a meridian $\mu$ of $T_{2,k}$ to a Seifert framing of $K$ and a longitude of $T_{2,k}$ to a meridian of $K$.

 Taking the 2-fold branched cover lifts $T_{2,k}$ to an unknot and the meridian $\mu$ to $\Tilde{\mu}\cap\tau\Tilde{\mu}$. Combining the discussion of the satellite operation above, we have
 \begin{center}
     $\Sigma_{2}(P(K))\cong (S^{3}_{k}(J)-N(\Tilde{\mu}-N(\tau\Tilde{\mu})){\cup}_{\partial N\Tilde{(\mu})}(S^{3}-N(K)){\cup}_{\partial N(\tau\Tilde{\mu})}(S^{3}-N(K))$,
 \end{center}
which is illustrated in Figure \ref{fig3}. Since $J$ is unknot, we have
\begin{center}
    $\Sigma_{2}(P(K))\cong S^{3}_{k}(K\#K^{r})$
\end{center}

\begin{figure}[ht]
 
\centering
\includegraphics[scale=0.7]{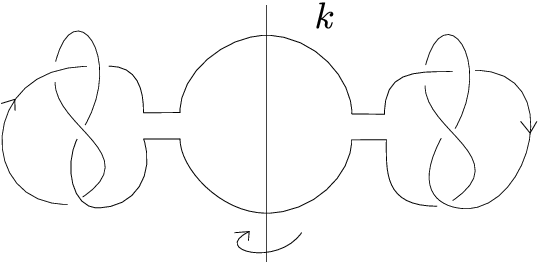}
\caption{knot surgery description}
\label{fig3}
 
\end{figure}

Note that, since it is a $k$ surgery on a knot in $S^{3}$, we have
\begin{center}
    $|H_{1}(\Sigma_{2}(P(K))|=k$.
\end{center}

\section{Proof of Theorem}
\subsection{knot Floer complex}\label{cfk}
We will give a description of ${CFK}^{\infty}(K\#K^{r})$ in this subsection. For any knot $K$, we have a filtered chain homotopy equivalence
\begin{center}
    ${CFK}^{\infty}(K\#K^{r})\cong {CFK}^{\infty}(K)\otimes {CFK}^{\infty}(K^{r})$.
\end{center}
from the connected sum formula \cite{ozsvath2004holomorphic}.

Since ${CFK}^{\infty}(K^{r})\cong {CFK}^{\infty}(K)$, we have that ${CFK}^{\infty}(K\#K^{r})\cong {CFK}^{\infty}(K)\otimes {CFK}^{\infty}(K)$.

When $K$ is a L-space knot, the knot Floer complex is in a relatively simple form
\begin{center}
    ${CFK}^{\infty}(K)\cong \mathrm{St}(K)\otimes \mathbb{Z}_{2}[U,U^{-1}]$.
\end{center}

Here, $\mathrm{St}(K)$ is the staircase complex associated to $K$. Figure \ref{fig5} is $\mathrm{St}(T_{3,4})$, where each dot represents a generator and the arrows represent differentials in the complex. The other complex in Figure \ref{fig5} is the tensor complex, we omit some differentials induced from the second components for simplicity.

Following the notation from section 4.1 of \cite{borodzik2014heegaard}, we can also denote this staircase complex by an array St(1,2,2,1). Each integer here denotes the length of the segments starting at the top left and moving to the bottom right in alternating right and downward steps. For a L-space knot $K$, the Alexander polynomial is in the form of $\Delta_{K}(t)=\Sigma_{i=0}^{2m}(-1)^{i}t^{n_{i}}$. We can get the staircase complex from the Alexander polynomial by $\mathrm{St}(K)=\mathrm{St}(n_{i+1}-n_{i})$, where $i$ runs from 0 to $2m-1$.

The absolute grading of the generator gives us a filtration on the staircase complex. The generator which does not have arrows pointing to other generators has grading 0 and we call these generators type \textbf{A}. Starting from top left, we denote these generators by $a_{1}$,$a_{2}$, $\dots$,$a_{m+1}$. Similarly, we call the other generators which have nontrivial differentials type \textbf{B} and denote them by $b_{1}$,$b_{2}$, $\dots$, $b_{m}$.

In the tensor product ${CFK}^{\infty}(K)\otimes {CFK}^{\infty}(K)$, we have a subcomplex $C^{'}$, which is generated by the concatenation of $a_{1}\otimes \mathrm{St}(K)$ and $\mathrm{St}(K)\otimes a_{n}$. We call the concatenation staircase the double of original staircase and denote it by $\mathrm{D}(\mathrm{St}(K))$. As an example, $\mathrm{D}(\mathrm{St}(T_{3,4}))$ is the red staircase in Figure \ref{fig5}.

\begin{figure}[ht]
 
\centering
\includegraphics[scale=0.7]{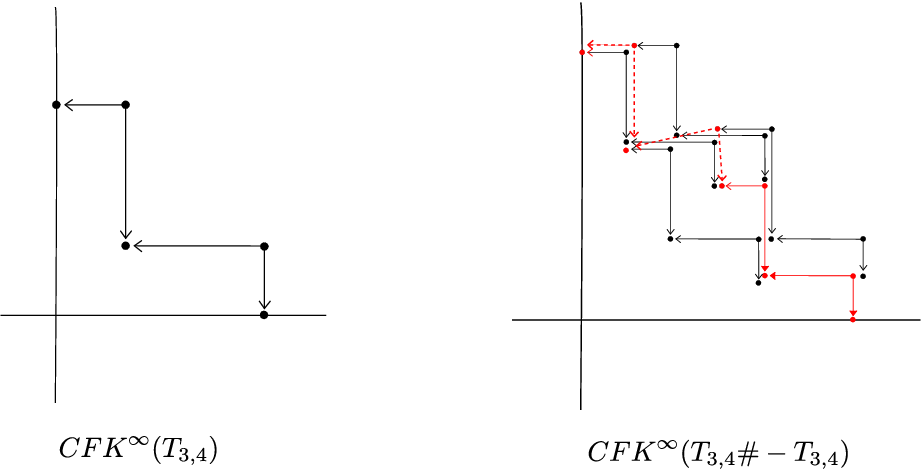}
\caption{knot Floer complex}
\label{fig5}
 
\end{figure}

Let $\Tilde{C}:={CFK}^{\infty}(K\#K^{r})/C^{'}$ be the quotient complex.

\begin{prop}\label{acyc}
    $H(\Tilde{C})\cong 0$.
\end{prop}
\begin{proof}
    We prove it by inductively quotienting the sub square complex from $\Tilde{C}$. At each generator $b_{i}b_{j}$, we have the following square complex \ref{figsquare} as a subcomplex of ${CFK}^{\infty}(K\#K^{r})$:
    \begin{figure}[ht]
    
    \centering
    \includegraphics[scale=0.2]{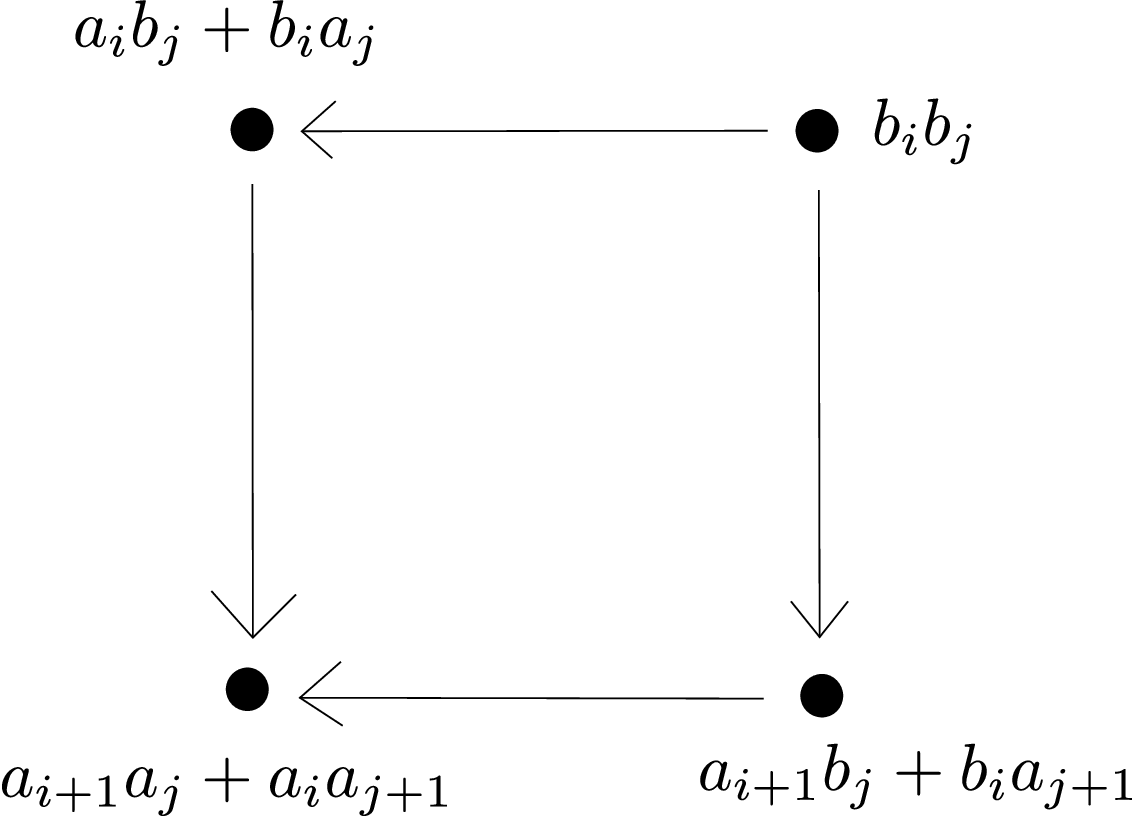}
    \caption{square complex}
    \label{figsquare}
 
    \end{figure}
    In $\Tilde{C}$, when $i=1$ or $j=m$, the square complexes become the ones in \ref{figsq}.
    \begin{figure}[ht]
    
    \centering
    \includegraphics[scale=0.6]{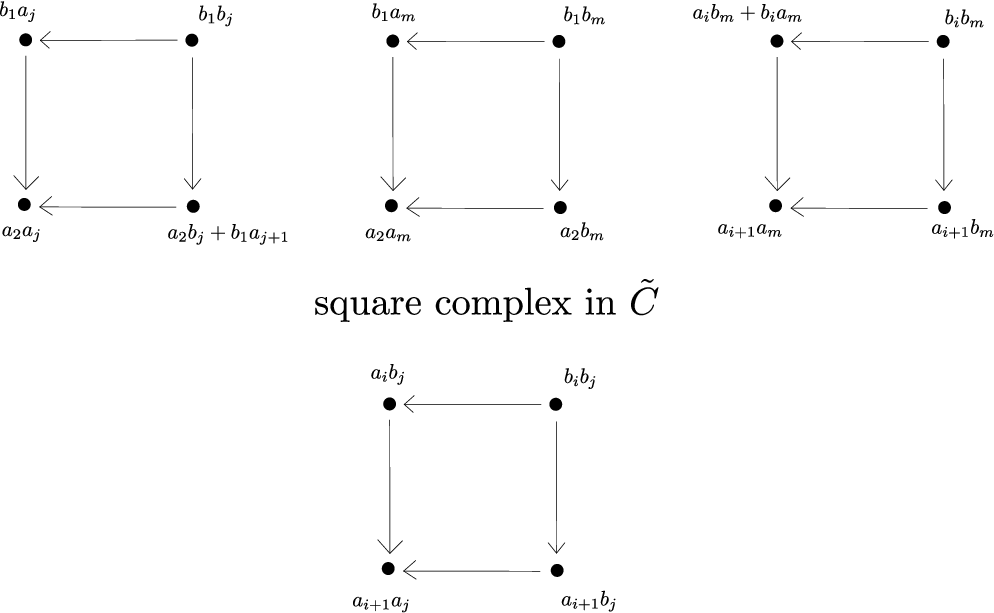}
    \caption{square complex in $\Tilde{C}$}
    \label{figsq}
 
    \end{figure}\\
    \\
    Quotienting these square complexes and performing a change of basis, the quotient complex we get is the same as deleting the square complex in the bottom of Figure \ref{figsq} with $i=1$ or $j=m$. Let us denote the new quotient complex by $\Tilde{C}_{1}$.

    Suppose we have quotiented $k$ times and got the quotient complex $\Tilde{C}_{k}$. By the same argument above, we can quotient the sub square complex with $i=k+1$ or $j=m-k$, which is the same as deleting the subcomplex in Figure \ref{figsq} with $i=k+1$ or $j=m-k$ and get the new quotient complex $\Tilde{C}_{k+1}$. Moreover, we have $\Tilde{C}_{m}=0$ since we have deleted all the generators.

    Hence, we have $\Tilde{C}\cong \Tilde{C}_{m}\cong 0$. 
\end{proof}

Combining Proposition \ref{acyc} and the exact sequence for the pair $({CFK}^{\infty}(K\#K^{r}),\mathrm{D}(\mathrm{St}(K))\otimes \mathbb{Z}_{2}[U,U^{-1}])$, we have the following:
\begin{prop}\label{cong}
    $H({CFK}^{\infty}(K\#K^{r}))\cong H(\mathrm{D}(\mathrm{St}(K))\otimes \mathbb{Z}_{2}[U,U^{-1}])$.
\end{prop}

The middle generator $a_{1}a_{n+1}$ of the double staircase is on the diagonal. This is easy to be shown since the staircase of any L-space knot is symmetric along the diagonal. The $i$-th filtration level of $a_{1}a_{n+1}$ is the distance of $a_{1}a_{n+1}$ to the $j$ axis, which is equal to sum of length of horizontal arrows in the staircase: $\Sigma_{i\in 2Z+1}n_{i}$. Since $a_{1}a_{n+1}$ is on the diagonal, we have
\begin{prop}\label{mid}
    The bigrading of $a_{1}a_{n+1}$ is $(\Sigma_{i\in 2Z+1}n_{i},\Sigma_{i\in 2Z+1}n_{i})$.
\end{prop}

\subsection{Computation of d invariants}
We will use the argument from \cite{ni2015cosmetic} to compute the d invariants.

First, we recall the integer surgery formula \cite{ozsvath2008knot}. Given a knot $K$ in $S^{3}$, let $C=CFK^{\infty}(K)$ be the knot Floer complex associated to it. The integer surgery formula asserts that we can compute $CF^{+}(S^{3}_{k},[i])$ from $C=CFK^{\infty}(K)$ in the following way.

There are two kinds of subcomplexes of $C=CFK^{\infty}(K)$. Define $A^{+}_{s}=C\{\mathrm{max}(i,j-s)\geq0\}$ and $B^{+}=C\{i\geq0\}$.

There are two canonical chain maps $v^{+}_{s}:A^{+}_{s}\rightarrow B^{+}$ and $h^{+}_{s}:A^{+}_{s}\rightarrow B^{+}$ as in \cite{ozsvath2008knot}. We only need $v^{+}_{s}$ in this paper, which is the projection from $A^{+}_{s}$ onto $C\{i\geq0\}$.

Let $\mathbb{A}=\otimes_{s\in\mathbb{Z}}A^{+}_{s}$ and $\mathbb{B}=\otimes_{s\in\mathbb{Z}}B^{+}_{s}$ and let $D^{+}_{n}:\mathbb{A}^{+}\rightarrow\mathbb{B}^{+}$ be the map
\begin{center}
    $D_{n}^{+}(\{a_{s}\}_{s\in\mathbb{Z}})=\{b_{s}\}_{s\in\mathbb{Z}},$
\end{center}
where here
\begin{center}
    $b_{s}=h^{+}_{s-k}(a_{s-n}+v^{+}_{s}(a_{s}).$
\end{center}

Let $\mathbb{X}^{+}(k)$ denote the mapping cone of $D^{+}_{k}$.

\begin{theor}\cite{ozsvath2008knot}
    For any non-zero integer $k$, the homology of the mapping cone $\mathbb{X}^{+}_{k}$ of
    \begin{center}
        $D^{+}_{k}:\mathbb{A}^{+}\rightarrow\mathbb{B}^{+}$
    \end{center}
    is isomorphic to $HF^{+}(S^{3}_{k}(K))$.
\end{theor}

In \cite{ni2015cosmetic}, Ni and Wu gave an efficient way to compute the d invariants from the integer surgery formula. We first recall the notation from their paper.

Let
\begin{center}
    $\mathfrak{A}_{s}^{+}=H_{*}(A^{+}_{s}),$   $\mathfrak{B}^{+}=H_{*}(B^{+}).$
\end{center}

Indeed, $B^{+}=C\{i\geq0\}$ is identified with $CF^{+}(S^{3})$ and $\mathfrak{B}^{+}\cong \mathcal{T}^{+}$. Here $\mathcal{T}^{+}\cong \mathbb{Z}_{2}[U,U^{-1}]/\mathbb{Z}_{2}[U]$. Let
\begin{center}
    $\mathfrak{v}^{+}_{s},\mathfrak{b}^{+}_{s}:\mathfrak{A}_{s}^{+}\rightarrow\mathfrak{B}^{+}$
\end{center}

be the map induced on homology.

Let $\mathfrak{A}^{T}_{s}=U^{n}\mathfrak{A}^{+}_{s}$ for $n\gg0$, we have $\mathfrak{A}^{T}_{s}\cong\mathcal{T}^{+}$. Since each $\mathfrak{a}_{s}^{+}$ is a graded isomorphism at sufficiently high grading and is $U$-equivariant, $\mathfrak{a}^{+}_{s}\mid\mathfrak{A}_{s}^{T}$ is modeled on multiplication by $U^{V_{s}}$. Note that the number $V_{s}$ is an invariant. Also, by Proposition \ref{cong}, we can use $\mathrm{D}(\mathrm{St}(K))\otimes \mathbb{Z}_{2}[U,U^{-1}]$ to compute $V_{s}$. We have a useful property of $V_{s}$.

\begin{prop}
    \cite{ni2015cosmetic}\cite{rasmussen2004lens}
   $V_{s}\geq V_{s+1}$.
\end{prop}

The formula given in \cite{ni2015cosmetic} computes d invariants of 3-manifold constructed from a rational surgery in $S^{3}$. In our case, we just need the formula in the integer surgery case.

\begin{prop}\cite{ni2015cosmetic}
    Suppose $k>0$ and fix $0\leq i\leq k-1$. Then
    \begin{center}
        $d(S^{3}_{k}(K),i)=d(L(k,1),i)-2\mathrm{max}\{V_{i},V_{k-i}\}$.
    \end{center}
\end{prop}

Combining this and the Proposition above, together with the symmetry of the d invariants for lens space, we have
\begin{center}
    $d(S^{3}_{k},i)=d(S^{3}_{k},k-i)=d(L(k,1),i)-2V_{i}$ , when $0\leq i\leq (k-1)/2$.
\end{center}

\begin{lem}\label{homology lem}
    For a staircase $\mathrm{St}\subseteq C$ and subcomplexes $A^{+}_{s}$ and $B^{+}$, let us denote the restriction of $\mathrm{St}$ to the subcomplexes by $r(\mathrm{St})$. $H_{*}(r(\mathrm{St}))$ is nontrivial iff $\mathrm{St}$ is fully included in the subcomplex.
\end{lem}
\begin{proof}
Each staircase in the subcomplex is truncated by a horizontal line and a vertical line. Suppose it is not fully included in the subcomplex, since the staircase starts horizontally and ends vertically, each connected component of the remaining part is a staircase with an even number of generators. Hence, the homology will be trivial on these staircases.  Below in Figure \ref{truncation}, we have $A^{+}_{3}(\mathrm{D}(\mathrm{St}(T(3,4)))\otimes\mathbb{Z}_{2}[U,U^{-1}])$ as an example.
\end{proof}

\begin{figure}[ht]
    
    \centering
    \includegraphics[scale=0.7]{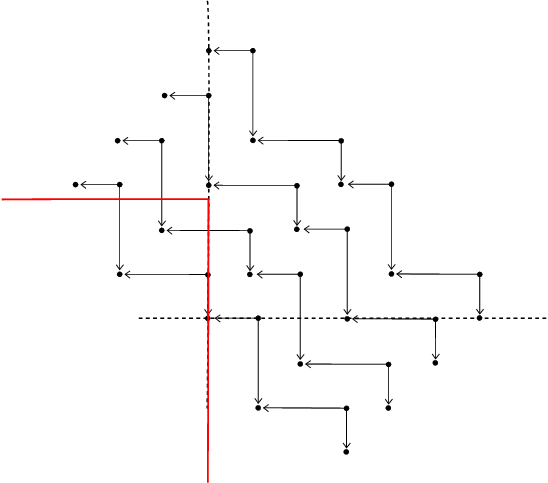}
    \caption{$A^{+}_{3}(\mathrm{D}(\mathrm{St}(T(3,4)))\otimes\mathbb{Z}_{2}[U,U^{-1}])$}
    \label{truncation}
 
    \end{figure}

Let us first study $\mathfrak{B}^{+}$. By Lemma \ref{homology lem}, the bottom generator is represented by the staircase whose left top corner is on the $j$-axis, since it is the first staircase which is fully included in $B^{+}$. Let us denote it by $St^{B}$.

The first staircase included in $\mathfrak{A}^{+}_{0}$ is the one that has the middle generator at $(0,0)$, let us denote it by $St^{0}$. $V_{0}$ is the $U$-distance between these two staircases, i.e. the $U$ power in $U^{V_{0}}St^{B}=St^{0}$. By Proposition \ref{mid}, the bigrading of the middle term in $\mathrm{St}^{B}$ is $(\Sigma_{i\in 2Z+1}n_{i},\Sigma_{i\in 2Z+1}n_{i})$. Since the middle generator of $St^{0}$ is at $(0,0)$, we get $V_{0}=\Sigma_{i\in 2Z+1}n_{i}$.

Let us denote the gap $V_{s}-V_{0}$ by $\bar{V}_{s}$. Note that when $s\geq 2\Sigma_{i\in 2Z+1}n_{i}$, $V_{s}=V_{0}$ and $\bar{V}_{s}=0$. On the j axis, we denote the overlap with all the staircases by $O_{i=0}$ and the restriction of $O_{i=0}$ to $0\leq j\leq s$ by $O_{i=0}^{s}$. We also denote the length of $O_{i=0}^{s}$ by $L_{i=0}^{s}$.

\begin{prop}
    When $0\leq s\leq 2\Sigma_{i\in 2Z+1}n_{i}$, $\bar{V}_{s}=s-L_{i=0}^{s}$.
\end{prop}
\begin{proof}
Let us look at the gap of $V_{s}-V_{s+1}$. Suppose $O_{i=0}^{s+1}\setminus O_{i=0}^{s}$ is nonempty, then the bottom most staircase of $A^{+}_{s}$ remains fully included in $A^{+}_{s+1}$. Hence, $V_{s}=V_{s+1}$ and $V_{s}-V_{s+1}=0$. Suppose $O_{i=0}^{s+1}\setminus O_{i=0}^{s}$ is empty, then the bottom most staircase of $A^{+}_{s+1}$ is the one which is once above the bottom most staircase of $A^{+}_{s}$. Hence $V_{s}-V_{s+1}=1$.

Sum all of these gaps up, we get the conclusion.
\end{proof}

Combine the propsitions above, we have
\begin{cor}\label{d invariants}
    $\bar{d}(S^{3}_{k}(K,s)=\bar{d}(L(k,1),s)-s+L_{i=0}^{s}$, when $0\leq s\leq (k-1)/2$.
\end{cor}

\subsection{Main theorem}\label{Proof}
In this subsection, we give the statement and the proof of the main theorem.

Let us denote $\mathrm{Max}\{s|s-L^{s}_{i=0}=0\}$ by $m(K)$. From the discussion about the staricase complex in Section \ref{cfk}, $m(K)$ is equal to the difference of degrees of the top degree generator and the second top degree generator of $\Delta_{K}(t)$. For any polynomial $P(t)=\sum_{i=1}^{n}a_{i}t^{d_{i}}$ such that $d_{i}<d_{i+1}$, let $m(P(t))=d_{n}-d_{n-1}$. Then $m(K)=m(\Delta_{K}(t))$.

It is been shown in \cite{hedden2018geography} that, for any L space knot $K$,
\begin{center}
    $\Delta_{K}(t)=t^{g}-t^{g-1}\cdots-t^{1-g}+t^{-g}$,
\end{center}
where $g$ denotes the Seifert genus of $K$. Hence, we have $m(K)=1$ for any L spcae knot $K$.

\begin{theor}\label{main}
    For any L space knot $K$, $K_{2,k_{1}}\#-T_{2,k_{1}}\#-{K}_{2,k_{2}}\#T_{2,k_{2}}$ has infinite order in the smooth knot concordance group $\mathcal{C}$ when $k_{1}$ and $k_{2}$ are a pair of distinct prime numbers such that $k_{1}>3$ and $k_{2}>3$.
\end{theor}

Note that, since $m(K)=1$, the assumption in Theorem \ref{main} is the same as $k_{i}>2m(K)+1$.

In the special case, when $K_{2,k_{1}}$ and $K_{2,k_{2}}$ are algebraic knots, $k_{i}>2pq>3$, for some $p,q$ which are coprime. We have the following corollary:

\begin{cor}
    $K_{2,k_{1}}\#-T_{2,k_{1}}\#-{K}_{2,k_{2}}\#T_{2,k_{2}}$ has infinite order in the smooth knot concordance group $\mathcal{}{C}$ when $K_{2,k_{1}}$ and $K_{2,k_{2}}$ are 2 distinct algebraic knots.
\end{cor}

Before proving the main theorem, let us study the linking form on the manifold $M=\Sigma_{2}(K_{2,k})\#-\Sigma_{2}(T_{2,k})$ first. From Section \ref{top}, we have $M\cong S^{3}_{k}(K\#K^{r})\#-S^{3}_{k}(U)$, here $U$ denotes the unknot. $H_{1}(M)\cong \mathbb{Z}/k\mathbb{Z}\oplus \mathbb{Z}/k\mathbb{Z}$, which is generated by the meridians of the surgery knots. Let us denote the meridian of $K\#K^{r}$ by $\alpha$ and the meridian of $U$ by $\beta$. Then the linking form evaluating on these generators gives:
\begin{center}
    $\lambda(\alpha,\alpha)\equiv1/k$, $\lambda(\beta,\beta)\equiv-1/k$, $\lambda(\alpha,\beta)\equiv0$. (mod $Z$)
\end{center}

Let us denote the generators of $H_{1}(nM)$ by $\alpha_{i}$ and $\beta_{i}$, $1\leq i\leq n$. Here, $\alpha_{i}$ are the meridians of each $K\#K^{r}$ and $\beta_{i}$ are the meridians of each $U$. We also use $M_{i}^{1}$ and $M_{i}^{2}$ to denote the corresponding summands of $\Sigma_{2}(K_{2,k})$ and $\Sigma_{2}(T_{2,k})$. Via a change of basis, we can use $\alpha_{i}$ and $\alpha_{i}+\beta_{i}$ as the generators for $H_{1}(nM)$.

For a knot $K$ satisfies the assumption in Theorem \ref{main}, we have
\begin{lem}\label{1}
    For any $i$, on the $\mathrm{Spin}^{c}$ structures correspond to the subgroup $G_{i}$ generated  by $\alpha_{i}+\beta_{i}$, there exists at least one $\mathrm{Spin}^{c}$ structure $\mathfrak{s}$, such that $\Bar{d}(nM,\mathfrak{s})\neq0$.
\end{lem}

\begin{proof}
    We prove it by contradiction. $G_{i}=\{l(\alpha_{i}+\beta_{i})|l\in \mathbb{Z}/k\mathbb{Z}\}$. Suppose for each $l$, $\Bar{d}(nM,l(\alpha_{i}+\beta_{i}))=0$. Then
    \begin{center}
        $\sum_{l=1}^{k}\Bar{d}(nM,l(\alpha_{i}+\beta_{i}))=0$.
    \end{center}
    Using additivity of the relative d invariants, we can rewrite it as
    \begin{center}
        $\sum_{l=1}^{k}\Bar{d}(M^{1}_{i},l\alpha_{i})=\sum_{l=1}^{k}\Bar{d}(M^{2}_{i},l\beta_{i})$.
    \end{center}
    By Corollary \ref{d invariants}, $\Bar{d}(M^{1}_{i},l\alpha_{i})=\Bar{d}(M^{2}_{i},l\beta_{i})-l+L_{i=0}^{l}$, when $0\leq l\leq (k-1)/2$.\\
    Note that $-l+L_{i=0}^{l}\leq0$ for any $l$. When $k>2m(K)+1$, $-(k-1)/2+L_{i=0}^{(k-1)/2}<0$ by the assumption. Hence, we have
    \begin{center}
        $\sum_{l=1}^{k}\Bar{d}(M^{1}_{i},l\alpha_{i})<\sum_{l=1}^{k}\Bar{d}(M^{2}_{i},l\beta_{i})$,
    \end{center}
    which contradicts the equation above.
\end{proof}

Let us denote the subgroup generated by $\alpha_{i}+\beta_{i}$ by $\Tilde{G}$.

\begin{lem}\label{2}
    For any metabolizer $G$ of $H_{1}(nM)$ with vanishing relative d invariants, $G\cap\Tilde{G}\neq\emptyset$.
\end{lem}

\begin{proof}
    For any metabolizer $G$, we have ${|G|}^{2}=k^{2n}$. Hence, ${|G|}=k^{n}$ and $G$ is generated by $n$ linearly independent elements $\{g_{j}=\sum_{i=1}^{n}a_{ij}\alpha_{i}+b_{ij}(\alpha_{i}+\beta_{i})|j=1,2,\cdots,n\}$.

    Suppose $G\cap\Tilde{G}=\emptyset$, then $G\cong G/G\cap\Tilde{G}$, which is generated by ${g_{j}^{'}=\sum_{i=1}^{n}a_{ij}\alpha_{i}}$. Since ${|G|}=k^{n}$, \{$g_{j}^{'}$\} are linearly independent. Hence $\alpha_{1}\in G/G\cap\Tilde{G}$, which implies there exists $ e=\alpha_{1}+\sum_{i=1}^{n}c_{i}(\alpha_{i}+\beta_{i})\in G$ for some $c_{i}\in \mathbb{Z}/k\mathbb{Z}$.

    Suppose $c_{1}\neq 0$, $e=(1+c_{1})\alpha_{1}+c_{1}\beta_{1}+\sum_{i=2}{n}(\alpha_{i}+\beta_{i})$. Use the same argument from Lemma \ref{1},
    \begin{center}
        $\sum_{l=1}^{k}\bar{d}(nM,le)\neq0$,
    \end{center}
    which contradicts the assumption.

    Suppos $c_{1}=0$, $e=\alpha_{1}+\sum_{i=2}{n}(\alpha_{i}+\beta_{i})$.
    \begin{center}
        $\bar{d}(nM,le)=\bar{d}(\Sigma_{2}(K_{2,k}),l)+\sum_{i=2}^{n}\bar{d}(M_{i},lc_{i}(\alpha_{i}+\beta{i}))\not\equiv 0 (\rm{mod} Z)$,
    \end{center}
    which also contradicts the assumption. Hence $G\cap\Tilde{G}\neq\emptyset$.
    \end{proof}

Proof of Theorem \ref{main}:
\begin{proof}
    By Lemma \ref{2}, any metablizer $G$ of $H_{1}(nM)$ contains an element $e\in \Tilde{G}$. Then by Lemma \ref{1} and additivity of relative d invariants, there exists at least one $\mathrm{Spin}^{c}$ structures $\mathfrak{s}$, such that $\Bar{d}(nM,\mathfrak{s})\neq0$.
    This shows the nonexistence of a metabolizer for which the relative d invariants vanish. Then by Corollary \ref{Core cor}, it proves the nonsliceness of $n(K_{2,k_{1}}\#-T_{2,k_{1}}\#-{K}_{2,k_{2}}\#T_{2,k_{2}})$.
\end{proof}

Using the jump function for the Levine-Tristram signature of a knot, we can show the linearly independent of a set of knots.

For the knot $K$, let us use $r(K)=\{\theta_{l}\}$ to denote the set of numbers such that when evaluated at $\omega=e^{2\pi i \theta_{l}}$, the jump function is non-zero.

\begin{cor}
    Let $k_{i}$ be a set of distinct prime numbers such that, $k_{i}>2m(K)+1$ and $1/2k_{i}\notin r(K)$. Then the set of knots $\{T_{2,k_{i}},K_{2,k_{i}}\}$ are linearly independent in the concordance group $\mathcal{C}$.
\end{cor}
\begin{proof}
    Consider a linear combination
    \begin{center}
        $J=\sum_{i=1}^{N}n_{i}T_{2,k_{i}}+m_{i}K_{2,k_{i}}$.
    \end{center}
    Suppose $J$ is slice. Fix $l$, when we evaluate the jump function at $\omega=e^{2\pi i/2k_{l}}$. By the assumption, the only knots have non-zero jump function are $T_{2,k_{l}}$ and $K_{2,k_{l}}$, both having jump equal to -1. Hence, $n_{l}=-m_{l}$.
    \begin{center}
        $J=\sum_{i=1}^{N}m_{i}(K_{2,k_{i}}-T_{2,k_{i}})$.
    \end{center}
    Using the same argument above, $J$ is nonslice.
\end{proof}

\bibliographystyle{alpha}  
\bibliography{references}

\end{document}